\documentclass{amsart}

\def\OO{\mathcal{O}}
\def\ff{\frak}
\def\Spec{\mbox{\rm Spec}}

\def\inv{{\rm inv}}

\def\cl{\mbox{\rm cl}}
\def\cal{\mathcal}

\def\X{{\ff X}}
\def\supp{\mbox{\rm supp}}
\def\PP{{\mathbb{P}}^1_k}
\def\PPD{{\mathbb{P}}^1_D}

\usepackage{amssymb,amscd}
 \usepackage{amsthm}

\begin{document}

\title{On the geometry of Pr\"ufer intersections of valuation rings}

\author{Bruce Olberding}

\address{Department of Mathematical Sciences, New Mexico State University,
Las Cruces, NM 88003-8001}

\email{olberdin@nmsu.edu}

\begin{abstract}
Let     $F$ be a field, let $D$ be a subring of $F$ and let $Z$ be  an irreducible subspace of  the space  of all valuation rings between $D$ and $F$ that have quotient field $F$.  Then $Z$ is a locally ringed space whose ring of global sections is $A = \bigcap_{V \in Z}V$. 
All rings between $D$ and $F$ that are integrally closed in $F$ arise in such a way. 
Motivated by applications in areas such as multiplicative ideal theory and real algebraic geometry, a number of authors have formulated criteria for when  $A$ is a Pr\"ufer domain.  
We give    geometric criteria for when $A$ is a Pr\"ufer domain that reduce this issue to questions of prime avoidance. These criteria, which unify and extend a variety of different results in the literature, are framed in terms of morphisms of $Z$ into the projective line $\PPD$.   
\end{abstract}




\date{}

\newtheorem{theorem}{Theorem}[section]
\newtheorem{lemma}[theorem]{Lemma}
\newtheorem{proposition}[theorem]{Proposition}
\newtheorem{corollary}[theorem]{Corollary}
\newtheorem{example}[theorem]{Example}
\newtheorem{remark}[theorem]{Remark}
\newtheorem{definition}[theorem]{Definition}

\maketitle


\section{Introduction}

  A subring $V$ of a field $F$   is a {\it valuation ring  of $F$} if for each nonzero $x \in F$, $x$ or $x^{-1}$ is in $V$; equivalently,  the ideals of $V$ are linearly ordered by inclusion and $V$ has  quotient field $F$. Although  the ideal theory of valuation rings is
 straightforward,  
    an intersection of valuation rings in $F$ can be quite complicated. Indeed, by a theorem of Krull \cite[Theorem 10.4]{Mat},   every integrally closed subring of  $F$ is an intersection of valuation rings of $F$. 
    In this article, we describe a geometrical approach to determining when an  intersection $A$ of valuation rings of $F$ is a  {\it Pr\"ufer domain},  meaning that  for each prime ideal $P$ of $A$, the localization $A_P$ is a valuation ring of $F$.  Whether an intersection of valuation rings is Pr\"ufer is of consequence in  multiplicative ideal theory, where Pr\"ufer domains are of central importance,  and real algebraic geometry, where the real holomorphy ring is a Pr\"ufer domain that expresses properties of fields involving sums of squares; see the 
 discussion below.  Over the past eighty years,  Pr\"ufer domains have been  extensively studied from ideal-theoretic, homological and module-theoretic points of view; see for example \cite{FHP, FS, G, KZ, LM}. 
  
{Throughout the paper $F$ denotes a field, $D$ is a subring of $F$ that  need not have quotient field $F$, and $Z$ is a subspace of  the {\it Zariski-Riemann space} $\X$ of $F/D$, the space  of all valuation rings of $F$ that contain $D$.}
 The   topology on $\X$ is given by  declaring the basic open sets to be those of the form $\{V \in \X:t_1,\ldots,t_n \in V\}$, where $t_1,\ldots,t_n \in F$. We assume  for technical convenience  that $F \in Z$.  
With this notation fixed, the focus of this article is the {\it holomorphy ring}\footnote{This terminology is due to Roquette \cite[p.~362]{Ro}. Viewing $Z$ as consisting of places rather than valuation rings, the elements of $A$ are precisely the elements of $F$ that have no poles (i.e., do not have value infinity) at the places in $Z$.} 
  $A= \bigcap_{V \in Z}V$ of the subspace $Z$. Such a ring is integrally closed in $F$, and, as noted above, every ring   between $D$ and $F$ that is  integrally closed in $F$ occurs as the holomorphy ring of a subspace of $\X$.   
 In general it is difficult to determine the structure of $A$ from properties of $Z$, topological or otherwise; see \cite{OIrr, ONoeth, OSpringer}, where the  emphasis is on the case in which $D$ is a two-dimensional Noetherian domain with quotient field $F$. In this direction,  there are a number of results  that are concerned with when the holomorphy ring $A$ is a  {Pr\"ufer domain} with quotient field $F$. Geometrically, this is equivalent to $\Spec(A)$ being an affine scheme in $\X$. 
  Moreover, 
    by virtue of the Valuative Criterion for Properness,  $A$ is a Pr\"ufer domain with quotient field $F$ if and only if there are no nontrivial proper birational morphisms into the scheme $\Spec(A)$, an observation that motivates Temkin and Tyomkin's notion of Pr\"ufer algebraic spaces  \cite{Temkin}.  
 
 We show in this article that the morphisms of $Z$ (viewed as a locally ringed space) into the projective line $\PPD$  determine whether the holomorphy ring $A$ of $Z$  is a Pr\"ufer domain. A goal in doing so is to provide a unifying explanation for an interesting variety of results in the literature.  By way of motivation, and because we will refer to them later, we recall these results here.

\smallskip

(1) Perhaps the earliest result in this direction is due to Nagata \cite[(11.11)]{N}: When   $Z$ is finite, then the holomorphy ring $A$ of $Z$ is a Pr\"ufer domain with quotient field $F$.

\smallskip 

(2) Gilmer \cite[Theorem 2.2]{Gi} shows  
that when  $f$ is a nonconstant  monic polynomial over $D$  having no root in $F$ and  each valuation ring in $Z$  contains the set   $S:=\left\{{1}/{f(t)}:t \in F\right\}$, then $A$ is a Pr\"ufer domain with torsion Picard group and quotient field $F$. Rush \cite[Theorem 1.4]{Rush} has since generalized this by allowing the polynomial $f$ to vary with the choice of $t$, but at the (necessary) expense of requiring the rational functions in $S$ to have certain numerators other than $1$.  
  Gilmer was motivated by a special case of this theorem due to Dress  \cite{Dr}, which states that when the field $F$ is formally real (meaning that $-1$ is not a sum of squares), then the subring of $F$ generated by $\left\{{(1+t^2)^{-1}}:t \in F\right\}$ is a Pr\"ufer domain with quotient field $F$ whose set of valuation overrings is precisely the set of valuation rings of $F$ 
 for which $-1$ is not a square in the residue field.  In the literature of real algebraic geometry, the Pr\"ufer domain thus constructed is the {\it real holomorphy ring} of $F/D$. The fact that such rings are Pr\"ufer has a number of interesting consequences for real algebraic geometry and sums of powers of elements of $F$; see for example Becker \cite{Becker} and Sch\"ulting \cite{Sch}. 
 These rings are also the only known source of Pr\"ufer domains having finitely generated ideals that cannot be generated by two elements, as was shown by Sch\"ulting  \cite{Sch79} and  Swan \cite{Swan}; the related literature on this aspect of holomorphy rings is discussed in \cite{OR}.  The notion of existential closure leads to more general results on Pr\"ufer holomorphy rings in function fields. For references  on this generalization, see  \cite{OlbH}.

\smallskip

(3)  Roquette \cite[Theorem 1]{Ro} proves that when there exists a nonconstant monic polynomial $f \in A[T]$ which has no root in the residue field of  $V$ for each valuation ring $V \in Z$ (i.e., the residue fields are ``uniformly algebraically non-closed''), then $A$ is a Pr\"ufer domain with torsion Picard group and quotient field $F$. 
Roquette developed these ideas  as a general explanation for his Principal Ideal Theorem, which states that the ring of totally $p$-integral elements of a formally $p$-adic field is a {\it B\'ezout domain}; that is, every finitely generated ideal is principal \cite[p.~362]{Ro}.  In particular, if there is a bound on the size of the residue fields of the valuation rings in  $Z$, then $A$ is a B\'ezout domain \cite[Theorem 3]{Ro}. 
Motivated by just such a situation, Loper   \cite{Lo} independently  proved similar results  in order to apply them to the ring of integer-valued polynomials of a domain $R$ with quotient field $F$: Int$(R) = \{g(T) \in F[T]:g(R) \subseteq R\}$. 

\smallskip

(4) In \cite{OR} it is 
 shown that when the holomorphy ring $A$ of $Z$ contains a field  of cardinality greater than that of $Z$, then $A$ is a B\'ezout  domain. 
 
 \smallskip

In  this article we offer a  geometric explanation for these results that reduces all the arguments to a question of homogeneous prime avoidance in the projective line $\PPD := {\rm Proj}(D[T_0,T_1])$.  
 Nagata's theorem in (1)  reduces to the observation that a finite set of points of $\PPD$ is contained in an affine open subset of  $\PPD$. The  example in (4) is explained similarly by showing that a ``small'' enough set of points in $\PPD$ is contained in an affine open set. And finally, in cases (2) and (3), the condition on the residue fields guarantees that  the image of each $D$-morphism $Z \rightarrow \PPD$ is contained in the open affine subset $(\PPD)_{g}$, where $g$ is the homogenization of $f$.

 To frame things geometrically, we view $Z$ as a locally ringed space. Its structure sheaf $\OO_Z$ is defined for each nonempty open subset $U$ of $Z$ by $\OO_Z(U) = \bigcap_{V \in Z}V$, while the ring of sections of the empty set is defined to be trivial ring with $0 = 1$; thus $\OO_Z$ is the {\it holomorphy sheaf} of $Z$.  
   The restriction maps on $\OO_Z$ off the empty set are simply set inclusion, and  
  the stalks of $\OO_Z$ are  the valuation rings in $Z$.  
  The standing assumption that $F$ is one of the valuation rings in $Z$ guarantees that $Z$ is an irreducible space; irreducibility in turn guarantees that $\OO_Z$ is a sheaf. 
  (Note that since we are interested in the ring $A = \bigcap_{V \in Z}V$, the  assumption that $F \in Z$ is no limitation.) When considering irreducible  subspaces $Y$ of $\X$, we similarly treat $Y$ as a locally ringed space with structure sheaf defined in this way.

By a {\it morphism} we always mean a morphism in the category of locally ringed spaces. If $X$ and $Y$ are locally ringed spaces with fixed morphisms $\alpha:X \rightarrow \Spec(D)$ and $\beta:Y \rightarrow \Spec(D)$, then  a morphism  $\phi:X \rightarrow Y$ is a {\it $D$-morphism} if $\alpha = \beta \circ \phi$.  
A scheme $X$ is a {\it $D$-scheme} if  a morphism $\phi:X \rightarrow \Spec(D)$ is fixed.  
There is a morphism $\delta=(d,d^\#):Z \rightarrow \Spec(D)$  defined by letting $d$ be the continuous map that sends a valuation ring in $Z$ to its center in $D$, and by letting $d^\#:\OO_{{\rm Spec}(D)} \rightarrow d_*\OO_Z$ be the sheaf morphism defined for each open subset $U$ of $\Spec(D)$ by the set inclusion $d^\#_U:\OO_{{\rm Spec}(D)}(U) \rightarrow \OO_Z(d^{-1}(U))$.  Thus when considering $D$-morphisms from $Z $ to $X$, with $X$ a $D$-scheme, we always assume that the structure morphism $Z \rightarrow \Spec(D)$ is the one defined above.

\section{Morphisms into projective space} 

In this section we describe the $D$-morphisms of $Z$ into projective space by proving an analogue of the fact that morphisms from schemes into projective space are determined by invertible sheaves.  
 Our main technical device in describing such morphisms is the notion of a projective model, as defined in \cite[Chapter VI, \S 17]{ZS}.  Let $t_0,\ldots,t_n $ be nonzero elements of $F$, and for each $i =0,1,\ldots,n$, define $D_i = D[{t_0}/{t_i},\ldots,{t_n}/{t_i}]$ and $U_i = \Spec(D_i)$. 
Then the {\it projective model of $F/D$ defined by $t_0,\ldots,t_n$} is   $X = \{(D_i)_P:P \in \Spec(D_i), i =0,1,\ldots,n\}$.  
The projective model $X$ is a topological space whose basic open sets are of the form $\{R \in X:u_0,\ldots,u_m \in R\}$, where $u_0,\ldots,u_m \in F$, and which is covered by the open subsets $\{(D_i)_P:P \in U_i\}$, $i=0,1,\ldots,n$.  Define a sheaf $\OO_X$ of rings on $X$ for each nonempty open subset $U$ of $X$ by $\OO_X(U) = \bigcap_{R \in U}R$, and let  the ring of sections of the empty set be the trivial ring with $0 = 1$.  Since $X$ is irreducible, $\OO_X$ is a sheaf and hence 
 $(X,\OO_X)$ is a scheme, and in light of the following remark, it is a projective scheme.

\begin{remark} \label{closed immersion remark}
{If $X$ is a projective model defined by $n+1$ elements, then there is a closed immersion $X \rightarrow {\mathbb{P}}^n_D$.}  {\em 
For let $X$ be the projective model defined by $t_0,\ldots,t_n \in F$. 
For each $i = 0,1,\ldots,n$, let $b_i:D[{T_0}/{T_i},\ldots,{T_n}/{T_i}] \rightarrow D_i$ be the $D$-algebra homomorphism that sends ${T_j}/{T_i}$ to ${t_j}/{t_i}$, and let $a_i:\Spec(D_i) \rightarrow \Spec(D[{T_0}/{T_i},\ldots,{T_n}/{T_i}])$ be the induced continuous map of topological spaces. Then the scheme morphisms $(a_i,b_i): \Spec(D_i) \rightarrow \Spec(D[{T_0}/{T_i},\ldots,{T_n}/{T_i}])$ glue together to a morphism $\phi:X \rightarrow {\mathbb{P}}^n_D$ \cite[p.~88]{Hart}, which by virtue of the way it is constructed is a closed immersion \cite[Lemma 01QO]{Stacks}.}\end{remark}

Let $t_0,\ldots,t_n$ be nonzero elements of $F$, and let $X$ be the projective model of $F/D$ defined by $t_0,\ldots,t_n$.  For each valuation ring $V$ in $Z$, there exists $i =0,1,\ldots,n$ such that ${t_j}/{t_i} \in V$ for all $j$, and it follows that 
each valuation ring $V$ in  $Z$ {\it dominates} a unique local ring $R$ in the model $X$, meaning that $R \subseteq V$ and the maximal ideal of $R$ is contained in the maximal ideal of $V$.   
The {\it domination morphism} $\delta=(d,d^\#):Z \rightarrow X$ is defined by letting $d$ be the continuous map that sends a valuation ring in $Z$ to the local ring in $X$ that it dominates, and 
 by letting $d^\#:\OO_{X} \rightarrow d_*\OO_Z$ be the sheaf morphism defined for each open subset $U$ of $Z$ by the set inclusion $d^\#_U:\OO_{X}(U) \rightarrow \OO_Z(d^{-1}(U))$.

 Let  $\gamma:X \rightarrow {\mathbb{P}}^n_D$ be the closed immersion defined in Remark~\ref{closed immersion remark}, and let $\delta:Z \rightarrow X$ denote the domination morphism. Then we say that the $D$-morphism $\gamma \circ \delta$ is the 
 {\it morphism defined by $t_0,\ldots,t_n$}.  We show in Proposition~\ref{factors} that 
 each $D$-morphism $Z \rightarrow {\mathbb{P}}^n_D$ arises in this way. Our standing assumption that $F \in Z$ is used in a strong way here, in that the proposition relies on a lemma which shows that the $D$-morphisms from $Z$ into projective space are calibrated by the inclusion morphism $\Spec(F) \rightarrow Z$.


\begin{lemma} \label{generic point} Let  $\iota:\Spec(F) \rightarrow Z$ be the canonical morphism, let  $\phi=(f,f^\#):Z \rightarrow X$ and $\gamma=(g,g^\#):Z \rightarrow X$ be morphisms of locally ringed spaces, where $X$ is a separated scheme, and let $\eta = f(F)$.  Then $\phi = \gamma$ if and only if $\phi \circ \iota = \gamma \circ \iota$; if and only if 
$\eta= f(F) = g(F)$ and $f^\#_\eta = g^\#_\eta$.  
\end{lemma}

\begin{proof}  
Suppose that $\eta= f(F) = g(F)$ and  $f^\#_\eta = g^\#_\eta$. 
 Let $U$ be an affine open subset of $X$ containing $\eta$, and let $Y = f^{-1}(U)$. Then $Y$ is a locally ringed space with structure sheaf $\OO_Y$ defined for each open set $W$ in $Y$ by $\OO_Y(W) = \OO_Z(W)$.  
  We claim that $\phi|_Y = \gamma|_Y$.  Since $U$ is affine and $Y$ is a locally ringed space, the morphisms $\phi|_Y$ and $\gamma|_Y$ are equal if and only if $f^\#_U = g^\#_U$ \cite[Theorem 10.8, p.~200]{Holme}.  Now since $\OO_{Z}(Y) \subseteq \OO_{Z,F} = F$ and the restriction maps on the sheaf $\OO_Z$ are set inclusions, 
   we have that for each $s \in \OO_X(U)$,    $f^\#_U(s) = f^\#_\eta(s) = g^\#_\eta(s) = g^\#_U(s)$. Thus $f^\#_U = g^\#_U$, and  hence $\phi|_Y = \gamma|_Y$.   Finally, let $\{U_i\}$ be the collection of all affine open subsets of $X$ that contain $\eta$.  Then $\{f^{-1}(U_i)\}$ is a cover of $Z$, and we have shown that $\phi$ and $\gamma$ restrict to the same morphism on each of these open sets, so we conclude that $\phi = \gamma$. It is straightforward to verify that $\phi \circ \iota = \gamma \circ \iota$ if and only if 
$f(F) = g(F)$ and $f^\#_\eta = g^\#_\eta$, so the lemma follows. 
\end{proof}

\begin{proposition} \label{factors}  If $\phi:Z \rightarrow {\mathbb{P}}^{n}_D$ is a $D$-morphism, then there exist $t_0\ldots,t_n \in F$ such that $\phi$ is defined by $t_0,\ldots,t_n$.  
\end{proposition}

\begin{proof} Write $\phi = (f,f^\#)$,  let $\eta = f(F)$, and let $S = {\mathbb{P}}^n_D = {\rm Proj}(D[T_0,\ldots,T_n])$. 
 For each $i =0,\ldots,n$, let $U_i $ be the open affine set $ S_{T_i},$ so that $S = U_0 \cup \cdots \cup U_n$.   
 Let $\alpha=(a,a^\#):\Spec(F) \rightarrow S$
 be the composition of $\phi$ with the canonical morphism $\Spec(F) \rightarrow Z$, and note that for each $i$,  $a^\#_{U_i}(s) = f^\#_{S,\eta}(s)$ for all $s \in \OO_{S}(U_i)$.  
  Since  
$\alpha$ 
  is a morphism of schemes into projective $n$-space over $D$, there exist $t_0,\ldots,t_n \in F$ such that for each $i,j$, $f^\#_{U_i}\left({T_j}/{T_i}\right) = {t_j}/{t_i}$; see the proof of \cite[Theorem II.7.1, p.~150]{Hart}.  Let $X$ be the projective model of $F/D$ defined by $t_0,\ldots,t_n$. Then $t_0,\ldots,t_n$ can be viewed as global sections of   an invertible sheaf on $X$ that is the image of the twisting sheaf $\OO(1)$ of $S$.  There is then by \cite[Theorem 7.1, p.~150]{Hart} and its proof a unique $D$-morphism $\gamma = (g,g^\#):X \rightarrow S$ such that  
   $g^\#_{U} = f^\#_{U}$ for each open set $U$ of $S$ and   $g:X \rightarrow S$ is  
 the continuous map that for each $i =0,\ldots,n$ sends the equivalence class of  a prime ideal $P$ in $\Spec(D[{t_0}/{t_i},\ldots,{t_n}/{t_i}]) \subseteq X$ to the equivalence class of the  prime ideal $(f^\#_{U_i})^{-1}(P)$ in $U_i=\Spec(D[{T_0}/{T_i},\ldots,{T_n}/{T_i}])$.   
 Then, with $\delta=(d,d^\#):Z \rightarrow X$ the domination morphism, $\gamma \circ \delta:Z \rightarrow S$ is a $D$-morphism. Moreover, $g(d(F)) = g(F) = \eta = f(F)$ and (viewing $F$ as a point in both $X$ and $Z$), $(d^\# \circ g^\#)_{F} = d^\#_{F} \circ g^\#_{\eta} = f^\#_{\eta}$. Therefore, by Lemma~\ref{generic point}, $\phi = \gamma \circ \delta$. 
\end{proof}

         \begin{corollary} \label{morphism cor}
 Every $D$-morphism $\phi:Z \rightarrow {\mathbb{P}}^{n}_D$ lifts 
 to  a unique $D$-morphism $\widetilde{\phi}:\X \rightarrow {\mathbb{P}}^{n}_D$.  
\end{corollary}

\begin{proof} 
 Let $\phi:Z \rightarrow {\mathbb{P}}^n_D$ be a $D$-morphism. Then by Proposition~\ref{factors} there exists a projective model $X$ of $F/D$ and a $D$-morphism $\gamma:X \rightarrow {\mathbb{P}}^{n}_D$ such that $\phi = \gamma \circ \delta|_{Z}$, where $\delta:Z \rightarrow X$ is the domination map. Since $X$ is a projective model of $F/D$, each valuation ring in $\X$ dominates $X$, and hence $\delta:Z \rightarrow X$ extends to the domination morphism $\widetilde{\delta}:\X \rightarrow X$.  
Thus $\widetilde{\phi} = \gamma \circ \widetilde{\delta}$ lifts $\phi$. If  there is another morphism $\psi:\X \rightarrow {\mathbb{P}}^n_D$ that lifts $\phi$, then with  $\iota:\Spec(F) \rightarrow Z$  the canonical morphism, $\psi \circ \iota = \phi \circ \iota = {\widetilde{\phi}} \circ \iota$, so that by  Lemma~\ref{generic point}, $\psi= \widetilde{\phi}$.   
\end{proof}

\begin{remark} {\em By Lemma~\ref{generic point}, the
 $D$-morphisms $Z \rightarrow {\mathbb{P}}^n_D$ are determined by their composition with the morphism $\Spec(F) \rightarrow {\mathbb{P}}^n_D$.  Conversely, by Corollary~\ref{morphism cor}, each $D$-morphism $\Spec(F) \rightarrow Z$ lifts to a unique morphism $Z \rightarrow \X$. Thus the $D$-morphisms $Z \rightarrow {\mathbb{P}}^n_D$ are in one-to-one correspondence with the 
  $F$-valued points of ${\mathbb{P}}^n_D$.}  
        \end{remark}
  
 \section{A geometrical characterization of Pr\"ufer domains}

We show in this section  that if $Z$ has the property that the image of every $D$-morphism $Z \rightarrow \PPD$ of locally ringed spaces factors through an affine scheme, then the holomorphy ring $A$  of $Z$ is a Pr\"ufer domain. A special case in which this is satisfied is when  there is a   homogeneous polynomial $f(T_0,T_1)$ of positive degree $d$ such that the image of each such morphism is contained in $(\PPD)_f$. In this 
case, we show that the Pr\"ufer domain $A$ has torsion Picard group. 
 
\begin{theorem} \label{big affine} The ring $A$ is a Pr\"ufer domain with quotient field $F$ if and only if every 
$D$-morphism $Z \rightarrow \PPD$ factors through an affine scheme.
\end{theorem} 
\begin{proof}
Suppose $A$ is a Pr\"ufer domain, and let $\phi:Z \rightarrow \PPD$ be a $D$-morphism. 
By Proposition~\ref{factors}, there exists a projective model $X$ of $F/D$ and a $D$-morphism $\gamma:X \rightarrow \PPD$ such that $\phi = \gamma \circ \delta$, where $\delta:Z \rightarrow X$ is the domination morphism. 
  Since  $A$ is a Pr\"ufer domain with quotient field $F$,  every localization of $A$
 is a valuation  domain and hence dominates a local ring in  $X$. Since every valuation ring in $Z$ contains $A$, it follows 
  that $\phi$ factors through the affine scheme $\Spec(A)$.   

Conversely, suppose that every $D$-morphism $Z \rightarrow \PPD$ factors through an affine scheme. Let   $P$ be a prime ideal of $A$.  To prove that $A_P$ is a valuation domain with quotient field $F$, it suffices to show that for each $0 \ne t \in F$, $t \in A_P$ or $t^{-1} \in A_P$. Let $0 \ne t \in F$, and let $X$  be the projective model of $F/D$ defined by $1,t$. Then by Remark~\ref{closed immersion remark} there is a closed immersion of $X$ into $\PPD$. Let $\phi=(f,f^\#):Z \rightarrow \PPD$ be the  $D$-morphism that results from composing this closed immersion with  the domination morphism $Z \rightarrow X$. In particular, with $\nu = f(F)$, we have $f^\#_{\nu}(T_1/T_0) = t$ and $f^\#_{\nu}(T_0/T_1) = t^{-1}$.  

 By assumption  there is a ring $R$ and $D$-morphisms $\delta=(d,d^\#):Z \rightarrow \Spec(R)$ and $\gamma=(g,g^\#):\Spec(R) \rightarrow \PPD$ such that $\phi = \gamma \circ \delta$.  By replacing $R$ with its  image in $F$ under $d^\#_{\eta}$, where $\eta = d(F)$, we may assume by Lemma~\ref{generic point} that $R$ is a subring of $F$ and that $\delta$ is the domination morphism. Then since $R$ is the ring of global sections of $\Spec(R)$ and $A$ is the ring of global sections of $Z$, it follows that 
  $R \subseteq A$, and hence  $Q = R \cap P$ is a prime ideal of $R$.   Let $x = g(Q)$. Then $x  
\in (\PPD)_{T_0}$ or $x \in (\PPD)_{T_1}$. In the former case, $f^\#_{x}(T_1/T_0) = t$, and in the latter, $f^\#_{x}(T_0/T_1) = t^{-1}$.  But  $f^\# = d^\# \circ g^\#$ and $d^\#$ restricts on each nonempty open subset of $\Spec(R)$ to the inclusion mapping, so either $x \in (\PPD)_{T_0}$, so that  $t = f^\#_x(T_1/T_0) =  g^\#_x(T_1/T_0) \in R_Q \subseteq A_P$, or $x \in (\PPD)_{T_1}$, so that   $t^{-1} = f^\#_x(T_0/T_1) =  g^\#_x(T_0/T_1) \in R_Q \subseteq A_P$. This proves that $A$ is a Pr\"ufer domain with quotient field $F$.  
\end{proof} 

Nagata's theorem discussed in (1) of the introduction follows then from Prime Avoidance:

\begin{corollary} {\em (Nagata \cite[11.11]{N})} If $Z$ is a finite set, then $A$ is a Pr\"ufer domain with quotient field $F$. 
\end{corollary} 

\begin{proof}  Let $\phi:Z \rightarrow \PPD$ be a $D$-morphism. Then the image of $\phi$ in $\PPD$ is finite, so by Homogeneous Prime Avoidance \cite[Lemma 1.5.10]{BH}, there exists a homogeneous polynomial $f$ (necessarily of positive degree) in the irrelevant ideal  $(T_0,T_1)$ of $D[T_0,T_1]$ such that $f$ is not in the union of the finitely many homogeneous prime ideals  corresponding to the image of $Z$ in $\PPD$; i.e., the image of $\phi$ is contained in $(\PPD)_f$.  This subset is affine \cite[Exercise III.10, p.~99]{EH}, so 
 by Theorem~\ref{big affine}, $A$ is a Pr\"ufer domain with quotient field $F$. \end{proof}

In fact, when $Z$ is finite, then $A$ is a B\'ezout domain: If $M$ is a maximal ideal of $A$, then $A_M$ is a valuation domain, but since $Z$ is finite, $A_M = \bigcap_{V \in Z}VA_M$, which since $A_M$ is a valuation domain, forces $A_M = V$ for some $V \in Z$. Therefore, $A$ has only finitely many maximal ideals, so that every invertible ideal is principal, and hence $A$ is a B\'ezout domain.

In Theorem~\ref{big affine 2}, we give a criterion  for when $A$ is  a Pr\"ufer domain with torsion Picard group. In this case, the $D$-morphisms $Z \rightarrow \PPD$ not only factor through an affine scheme, but have image in an affine open subscheme of $\PPD$.  
For lack of a precise reference, we note the following standard observation.

\begin{lemma} \label{exercise} Let $X$ be a projective model of $F/D$ defined by $t_0,\ldots,t_n \in F$, and let $f(T_0,\ldots,T_n) \in D[T_0,\ldots,T_n]$ be homogeneous of positive degree $d$ such that $f(t_0,\ldots,t_n) \ne 0$.  Let  $$R = \{0\} \cup  \left\{\frac{h(t_0,\ldots,t_n)}{f(t_0,\ldots,t_n)^e}:e \geq 0 {\mbox{ and }} h {\mbox{ is a homogeneous form of degree }} de\right\}.$$ Then $\{R_P:P \in \Spec(R)\}$ is an open  affine subset of $X$.  
%
\end{lemma}

\begin{proof}  
 Let $S = {\mathbb{P}}^n_D$. Then $S_f$ is an open affine subset of $S$  \cite[Exercise III.10, p.~99]{EH}.  By Remark~\ref{closed immersion remark}, there is a closed immersion
 $\gamma=(g,g^\#):X \rightarrow S$ such that  
with $\eta = g(F)$, we have for each $i,j$,  $g^\#_{\eta}\left({T_j}/{T_i}\right) = {t_j}/{t_i}$.  
Since $S_f$ is an open affine subset of $S$ and $\gamma$ is a closed immersion, then   $g^{-1}(S_f)$ is an open affine subset of $X$ whose ring of sections is
$g^\#_{\eta}(\OO_{S}(S_f))$ \cite[Lemma  01IN]{Stacks}.  Now $\OO_S(S_f)$ is the ring consisting of $0$ and the rational functions of the form ${h}/{f^e}$, where $e>0$ and $h$ is a homogeneous form of degree $de$. 
 Moreover, for such a rational function, since $f(t_0,\ldots,t_n) \ne 0$,  we have that 
 $f(T_0,\ldots,T_n)$ is a unit in $\OO_{S,\eta}$ and 
 $$
 g^\#_{\eta}\left(\frac{h(T_0,\ldots,T_n)}{f(T_0,\ldots,T_n)^e}\right)  = \frac{h(t_0,\ldots,t_n)}{f(t_0,\ldots,t_n)^e} \in R.$$ Thus $g^\#_{\eta}(\OO_S(S_f)) = R$, which proves the lemma.   
\end{proof}


\begin{lemma} 
\label{power}
 Let $t_0,t_1,\ldots,t_n$ be nonzero elements of $F$, and  let $f$ be a homogeneous polynomial in $D[T_0,\ldots,T_n]$ of positive degree $d$. 
Then the following are equivalent.

\begin{itemize}


\item[{(1)}] $t_0^d,\ldots,t_n^d \in  f(t_0,\ldots,t_n)A$.  

\item[{(2)}]  $(t_0,\ldots,t_n)^d A = f(t_0,\ldots,t_n)A$.


\item[{(3)}] The image of the morphism $Z \rightarrow {\mathbb{P}}^n_D$ defined by $t_0,\ldots,t_n$ is in $({\mathbb{P}}^n_D)_f$.  
 \end{itemize} 
\end{lemma}

\begin{proof} Let $u = f(t_0,\ldots,t_n).$
First we claim that (1) implies (2). If $V \in Z$, then there is $i$ such that $t_i$ divides in $V$ each of $t_0,\ldots,t_n$.  It follows that  when $\sum_i e_i = d$ for nonnegative integers $e_i$, then $t_0^{e_0}t_1^{e_1} \cdots t_n^{e_n} \in t_i^dV$.
Thus by (1), $t_0^{e_0}t_1^{e_1} \cdots t_n^{e_n} \in uV$, so that $t_0^{e_0}t_1^{e_1} \cdots t_n^{e_n} \in uA$. Statement (2)  now follows.

    To see
 that (2) implies (3),   let $\gamma=(g,g^\#):Z \rightarrow {\mathbb{P}}^n_D$ be the morphism defined by $t_0,\ldots,t_n$.  By (2), $u = f(t_0,\ldots,t_n)$ is nonzero. 
Define  
\begin{eqnarray*} R & = & \{0\} \cup  \left\{\frac{h(t_0,\ldots,t_n)}{u^e}:e \geq 0 {\mbox{ and }} h {\mbox{ is a homogeneous form of degree }} de\right\} \\ S & = &  \{0\} \cup  \left\{\frac{h(T_0,\ldots,T_n)}{f(T_0,\ldots,T_n)^e}:e \geq 0 {\mbox{ and }}  h {\mbox{ is a homogeneous form of degree }} de\right\},\end{eqnarray*} so that $({\mathbb{P}}^n_D)_f = \Spec(S)$. 
Let $\alpha= (a,a^\#):\Spec(R) \rightarrow \Spec(S)$ be the morphism induced by the ring homomorphism $a^\#:S \rightarrow R$ given by evaluation at $t_0,\ldots,t_n$.  
We claim that $R \subseteq A$.
%
For let $h$ be a homogeneous form in $D[T_0,\ldots,T_n]$ of degree $de$. Then  by (2),  $h(t_0,\ldots,t_n) \in (t_0,\ldots,t_n)^{de}A=  u^{e}A$, so that $R \subseteq A$.   Now let $\beta:Z \rightarrow \Spec(R)$ be the induced domination morphism. We claim that  $\gamma = \alpha \circ \beta$.  Indeed, by Lemma~\ref{exercise}, $\Spec(R)$ is an affine submodel of the projective model $X$ of $F/D$ defined by $t_0,\ldots,t_n$, and $\gamma$ factors through $X$. Since $\beta$ is the domination mapping, it follows that $\gamma = \alpha \circ \beta$, and hence the image of $\gamma$ is contained in $\Spec(S) = ({\mathbb{P}}^n_D)_f$.
%

Finally, to see that (3) implies (1),  let $U =  ({\mathbb{P}}^n_D)_f$ and let $\gamma=(g,g^\#):Z \rightarrow {\mathbb{P}}^n_D$ be the morphism defined by $t_0,\ldots,t_n$. 
Since by (3), $Z \subseteq  g^{-1}(U)$, then $S$, the ring of sections of $U$, is mapped via  $g^\#_U$ into the holomorphy ring $A$ of $Z$.  But the image of $g^\#_U$ is $R$, so  $R \subseteq A$, and hence every element of $F$ of the form 
${t_i^{d}}/{u}$ is an element of $A$, from which (1) follows. 
\end{proof}

\begin{theorem} \label{big affine 2} The ring $A$  is a Pr\"ufer domain with torsion Picard group and quotient field $F$ if and only if for each $A$-morphism $\phi:Z \rightarrow {\mathbb{P}}_A^1$ there is a homogeneous polynomial  $f \in A[T_0,T_1]$ of positive degree such that the image of $\phi$ is in $({\mathbb{P}}_A^1)_f$.
\end{theorem} 

\begin{proof} 
The choice of the subring $D$ of $F$ was arbitrary, so for the sake of this proof we may assume without loss of generality that $D = A$ and apply then the preceding results to $A$. 
Suppose that for  each $A$-morphism $\phi:Z \rightarrow {\mathbb{P}}_A^1$ there exists a homogeneous polynomial $f \in A[T_0,T_1]$ of positive degree such that the image of $\phi$ is in the affine subset  $({\mathbb{P}}_A^1)_f$.
By Theorem~\ref{big affine}, 
 $A$ 
 is a Pr\"ufer  domain with quotient field $F$.   
Thus to prove that $A$ has torsion Picard group, it suffices to show that for each two-generated ideal  $(t_0,t_1)A$ of $A$, there exists $e>0$ such that $(t_0,t_1)^eA$ is a principal ideal (see for example the proof of \cite[Theorem 2.2]{Gi}). 
 Let $t_0,t_1 \in F$, and let $\phi:Z \rightarrow {\mathbb{P}}_A^1$ be the morphism defined by $t_0,t_1$.  Then by assumption, there exists a homogeneous polynomial $f \in A[T_0,T_1]$ of positive degree $d$ such that the image of $Z$ in ${\mathbb{P}}_A^1$ is contained in $({\mathbb{P}}_A^1)_f$.
 Thus by 
 Lemma~\ref{power}, $(t_0,t_1)^dA$ is a principal ideal.
 
 Conversely,  let $\phi:Z \rightarrow {\mathbb{P}}_A^1$ be an $A$-morphism. Then by Proposition~\ref{factors} there exist $t_0,t_1 \in F$ such that $\phi$ is defined by $t_0,t_1$. 
Since $A$ has torsion Picard group and quotient field $F$,  there exists $d>0$ such that $(t_0,t_1)^dA = uA$ for some $u \in (t_0,t_1)^dA$. Since $u$ is an element of $(t_0,t_1)^dA$, there exists a homogeneous polynomial $f \in A[T_0,T_1]$ of  degree $d$ such that $f(t_0,t_1) = u$, and hence by Lemma~\ref{power}, the image of the morphism $\phi$ is contained in $({\mathbb{P}}_A^1)_f$. 
 \end{proof}

For applications such as those discussed in  (2) and (3) of the introduction, one needs to work with $D$-morphisms into the projective line over $D$, rather than $A$.  This involves a change of base, but causes no difficulties when verifying that $A$ is a Pr\"ufer domain. However, the converse of Theorem~\ref{big affine 2} (which is  not needed in the applications in (2) and (3) of the introduction) is lost in the base change.

\begin{corollary} \label{main cor}  
If for each $D$-morphism $\phi:Z \rightarrow \PPD$ there exists a homogeneous polynomial $f \in D[T_0,T_1]$ of positive degree such that the image of $Z$ is contained in $(\PPD)_f$, then $A$ is a Pr\"ufer domain with torsion Picard group and quotient field $F$.  
\end{corollary}

\begin{proof} 
Let $\phi:Z \rightarrow {\mathbb{P}}_A^1$ be a $D$-morphism, and let $\alpha:{\mathbb{P}}_A^1 \rightarrow {\mathbb{P}}_D^1$ be the change of base morphism.  By assumption there exists a homogeneous polynomial $f \in D[T_0,T_1]$ such that the image of $\alpha \circ \phi$ is contained in $(\PPD)_f$. Then  the image of $\phi$ is contained in $({\mathbb{P}}_A^1)_f$, and the corollary follows from Theorem~\ref{big affine}. 
\end{proof}  

Let $n$ be a positive integer. An abelian group $G$ is an {\it $n$-group} if each element of $G$ has finite order and this order is divisible by such primes only which also appear as factors of $n$. If $A$ is a Pr\"ufer domain with quotient field $F$, then the Picard group of $A$ is an $n$-group if and only if for each $t \in F$ there exists $k>0$ such that $(A + tA)^{n^k}$ is a principal fractional ideal of $A$ \cite[Lemma 1]{Ro}.

\begin{remark} \label{exponent} {\em
   If each homogeneous polynomial $f$ arising as in the statement of the corollary  can be chosen with degree $\leq  n$ ($n$ fixed), then  the Picard group  of the Pr\"ufer domain $A$ is an $n$-group. For  when $t \in F$ and $\phi:Z \rightarrow \PPD$ is the $D$-morphism defined by $1,t$, then with $f$ the polynomial  of degree $\leq n$  given by the corollary,  Lemma~\ref{power} shows that 
 $(A+tA)^n$ is a principal fractional ideal of $A$. %
In particular, when for each $D$-morphism $\phi:Z \rightarrow \PPD$, there exists a linear homogeneous polynomial $f \in A[T_0,T_1]$  such that the image of $\phi$ is contained in $({\mathbb{P}}_A^1)_f$, then the ring $A$  is a {B\'ezout domain}  with quotient field $F$. } 
\end{remark} 

The next corollary is a stronger version of statement (4) in the introduction. 



\begin{corollary}  If $D$ is a local domain  and   $Z$ has cardinality less than that of the residue field of $D$, then $A$ is a B\'ezout domain with quotient field $F$.
\end{corollary}

\begin{proof}  
 Let $\phi:Z \rightarrow \PPD$ be a $D$-morphism. 
For each $P \in $ Proj$(D[T_0,T_1])$, let $\Delta_P = \{d \in D :T_0+dT_1 \in P\}$. Then all the elements of $\Delta_P$ have the same image in the residue field of $D$. Indeed, 
if $d_1,d_2 \in  \Delta_P$, then $(d_1 - d_2)T_1 = (T_0+d_1T_1)-(T_0+d_2T_1)  \in P$. 
If  $T_1 \in P$, then since $T_0+d_1T_1 \in P$, this forces $(T_0,T_1) \subseteq P$, a contradiction to the fact that $P \in $ Proj$(D[T_0,T_1])$. Therefore, $T_1 \not \in P$, so that $d_1 - d_2 \in P \cap D \subseteq {\ff m}:=$ maximal ideal of $D$, which shows that  all the elements of $\Delta_P$ have the same image in the residue field of $D$. Let $X$ denote the image of $\phi$  in $\PPD$.   Then since $|X| < |D/{\ff m}|$, there 
 exists $d \in D \smallsetminus \bigcup_{P \in X}\Delta_P$, and hence $f(T_0,T_1):=T_0+dT_1 \not \in P$ for all $P \in X$.    Thus the image of $\phi$ is   in $(\PPD)_f$, and  
by Corollary~\ref{main cor} and Remark~\ref{exponent},  
 $A$ is a B\'ezout 
domain with quotient field $F$. 
\end{proof}

The following corollary is a small improvement of a theorem of Rush  \cite[Theorem 1.4]{Rush}. Whereas the theorem of Rush requires that $1,t,t^2,\ldots,t^{d_t} \in f_t(t)A$, we need only that 
 $1,t^{d_t} \in f_t(t)A$.

\begin{corollary} \label{torsion Prufer} The ring   $A$ is a Pr\"ufer domain with torsion Picard group and quotient field $F$ if and only if  for each $0 \ne t \in F$, there is  a polynomial $f_t(T) \in A[T]$ of positive degree $d_t$ such that 
 $1,t^{d_t} \in f_t(t)A$.  
\end{corollary}

\begin{proof} 
 If $A$ is a Pr\"ufer domain with torsion Picard group and quotient field $F$, then for each $0 \ne t \in F$, there is  $d_t>0$ such that $(1,t)^{d_t}A$ is a principal fractional ideal of $A$. Since $A$ is a Pr\"ufer domain, local verification shows that $(1,t)^{d_t}A = (1,t^{d_t})A$, and it follows that  there is  a polynomial $f_t(T) \in A[T]$ of positive degree $d_t$ such that 
 $1,t^{d_t} \in f_t(t)A$.  
 
 To prove the converse, 
we use Theorem~\ref{big affine 2}.  
 Let $\phi:Z \rightarrow {\mathbb{P}}^1_D$ be  a $D$-morphism. Then by Proposition~\ref{factors}  there exists $0 \ne t \in F$ such that $\phi$ is defined by $1,t$. 
 By assumption, there is  a polynomial $f_t(T) \in A[T]$ of positive degree $d_t$ such that 
 $1,t^{d_t} \in f_t(t)A$.
 Set $g_t(T_0,T_1) = f_t(T_0/T_1)T_1^{d_t}$, so that $g_t(T_0,T_1)$ is a homogeneous form of positive degree. 
 Then $1,t^{d_t} \in g_t(t,1)A$, and  by  Lemma~\ref{power}
 the image of $\phi$ is in $({\mathbb{P}}^1_A)_g$. By Theorem~\ref{big affine 2}, 
  $A$ is a Pr\"ufer domain with torsion Picard group and quotient field $F$.  \end{proof}



Rush  \cite[Theorem 2.2]{Rush} proves that when  $f$ is a monic polynomial of positive degree in $A[T]$, then 
 (a) $\{{1}/{f(t)}:t \in F\} \subseteq A$ if and only if (b) the image of $f$ in $(V/{\ff M}_V)[T]$ has no root in $V/{\ff M}_V$ for each $V \in Z$; if and only if (c) $A$ is a Pr\"ufer domain and $f(a)$ is a unit in $A$ for each $a \in A$.   As Rush points out, Gilmer's theorem discussed in (2) of the introduction 
  follows quickly from the equivalence of (a) and (b) and Corollary~\ref{torsion Prufer}; see the discussion on pp.~314-315 of \cite{Rush}.  Similarly, the results of Loper and Roquette described in (3) of the introduction also follow from Corollary~\ref{torsion Prufer} and the equivalence of (a) and (b).   Thus all the constructions in (1)--(4) of the introduction are recovered by the results in this section.

\section{The case where $D$ is a local ring}

This section focuses on  the case where $D$ is a local ring that is integrally closed in $F$. (By a local ring, we mean a ring that has  a unique maximal ideal; in particular, we do not require local rings to be Noetherian.)   In such a case,  as is noted in the proof of Theorem~\ref{P1 theorem}, every proper subset of closed points of $\PPD$   is contained in an affine open subset of $\PPD$, a fact which leads to a stronger result than could  be obtained in the last section. To prove the theorem, 
we need a   coset version of homogeneous prime avoidance. The proof of the lemma follows Gabber-Liu-Lorenzini \cite{GLL} but involves a slight modification to permit  cosets.

\begin{lemma} \label{PA} {\em (cf.~\cite[Lemma 4.11]{GLL})} Let $R = \bigoplus_{i=0}^\infty R_i$ be a graded ring, and let $P_1,\ldots,P_n$ be incomparable  homogeneous prime ideals not containing $R_1$. Let $I=\bigoplus_{i=0}^\infty I_i$ be a homogeneous ideal of $R$ such that $I \not \subseteq P_i$ for each $i=1,\ldots,n$. Then there exists $e_0>0$ such that for all $e \geq e_0$ and $r_1,\ldots,r_n \in R$, $I_{e} \not \subseteq  \bigcup_{i=1}^n (P_i +r_i)$.  
\end{lemma}

\begin{proof}
The proof is by induction on $n$. For the case $n = 1$, let $s$ be a homogeneous element in $ I \smallsetminus P_1$, let $e_0 = \deg s$, let $e\geq e_0$ and let $t \in R_1 \smallsetminus P_1$.  
 Suppose that $r_1 \in R$ and  $I_{e} \subseteq P_1 + r_1$.  Then since $0 \in I_{e}$, this forces $r_1 \in P_1$ and hence $st^{e-e_0} \in I_{e} \subseteq P_1$, a contradiction to the fact that neither $s$ nor $t$ is in $P_1$. Thus $I_{e} \not \subseteq P_1 + r_1$. 
Next, let $n>1$,  and suppose that 
 the lemma holds for $n-1$. 
 Then since the $P_i$ are incomparable,   $IP_1 \cdots P_{n-1} \not \subseteq P_n$, and by the case $n=1$, there exists $f_0>0$ such that for all $f \geq f_0$ and $r_n \in R$, $(IP_1 \cdots P_{n-1})_{f} \not \subseteq  (P_n + r_n)$.    
 Also, by the induction hypothesis, there exists $g_0>0$ 
such that for all $g \geq g_0$ and $r_1,\ldots,r_{n-1} \in R$, $(IP_n)_{g} \not \subseteq  \bigcup_{i=1}^{n-1}(P_i+r_i)$.   
Let $e_0 = \max \{f_0,g_0\}$, let $e \geq e_0$ and let $r_1,\ldots,r_n \in R$.  
Then in light of the above considerations, we may choose $a \in (IP_1 \cdots P_{n-1})_{e} \smallsetminus (P_n + r_n)$ and $b \in (IP_n)_{e} \smallsetminus  \bigcup_{i=1}^{n-1}(P_i+r_i)$.   
Then $a + b \in I_{e} \smallsetminus \bigcup_{i=1}^n (P_i +r_i)$.
\end{proof}


\begin{theorem} \label{P1 theorem} Suppose $D$ is local and integrally closed in $F$ and 
 only finitely many valuation rings in $Z$  do not dominate $D$. 
  If 
no $D$-morphism $Z \rightarrow \PPD$ has every closed point of $\PPD$ in its image, then  
 $A$ is a Pr\"ufer domain with torsion Picard group and quotient field $F$.

\end{theorem}

 
 


  


\begin{proof}
Let $S = D[T_0,T_1]$. 
By Corollary~\ref{main cor} it suffices to show that for each $D$-morphism $\phi:Z \rightarrow \PPD$, there is a homogeneous polynomial $f \in S$ of positive degree such that the image of $\phi$ is in $(\PPD)_f$. To this end, let $\phi:Z \rightarrow \PPD$ be a $D$-morphism. By assumption, there is a closed point $x \in \PPD$ not in the image of $\phi$. 
Let $\pi:\PPD \rightarrow \Spec(D)$ be the structure morphism. Since $\pi$ is a proper morphism, $\pi$ is closed and hence 
 $\pi(x)$ is a closed point in $\Spec(D)$. Thus since $D$ is  local, $\pi(x)$ is the maximal ideal ${\ff m}$ of $D$.  Let $\Bbbk$ be the residue field of $D$. 
Then, with $Q$ the  
homogeneous prime ideal in $S$ corresponding  to $x$, we must have ${\ff m} \subseteq Q$, and hence Proj$(\Bbbk[T_0,T_1])$ is isomorphic to a closed subset of $\PPD$ containing $Q$. Since a homogeneous prime ideal in Proj$(\Bbbk[T_0,T_1])$ is generated by a homogeneous polynomial  in $\Bbbk[T_0,T_1]$, it follows that there is 
 a homogeneous polynomial  $g \in S$ of positive degree $d$ such that $Q = ({\ff m},g)S$.   
Since, as noted above, every prime ideal in $\PPD =$ Proj($S)$ corresponding to a closed point in $\PPD$ contains ${\ff m}$, it follows that every closed point in $\PPD$ distinct from $x$ is contained in $(\PPD)_g$. Thus if every valuation ring in $Z$ other than $F$ dominates $D$, then the image of $\phi$ is contained in $(\PPD)_g$, which proves the theorem. 

It remains to consider the case where $Z$ also contains, in addition to the valuation ring $F$,  
 valuation rings $V_1,\ldots,V_n$  that are not centered on the maximal ideal ${\ff m}$ of $D$. 
Let $P_1,\ldots,P_n$ be the homogeneous prime ideals of $S$ that are the images under $\phi$ of $V_1,\ldots,V_n$, respectively.   
Let $I = {\ff m}S$. Since no $V_i$  dominates $D$, then since $\phi$ is a morphism of locally ringed spaces,  $I \not \subseteq {P}_i$ for all $i =1,\ldots,n$. We may assume  $P_{1},\ldots,P_{k}$  are the prime ideals that are maximal in the set $\{P_1,\ldots,P_n\}$. 
Then by Lemma~\ref{PA}, there exists $e>0$ such that  $I_{de} \not \subseteq \bigcup_{i=1}^k(P_{i}+g^e)$. Let $h$ be a homogeneous element in $I_{de} \smallsetminus \bigcup_{i=1}^k(P_{i}+g^e).$ 
Since  $P_1,\ldots,P_k$  are maximal in $\{P_1,\ldots,P_n\}$, it follows that $h \in I_{de} \smallsetminus \bigcup_{i=1}^n(P_{i}+g^e).$
  Set $f = h-{g}^e$. Then
   $f \not \in P_i$ for all $i$. 
     In particular, $f \ne 0$, and hence
      $f$  is homogeneous of degree $de$. 
Since $f \not \in P_1 \cup \cdots \cup P_n$, then $P_1,\ldots,P_n \in (\PPD)_f$. 
  
  Finally we show that every closed point of $\PPD$ distinct from $x$ is in $(\PPD)_f$. Let $L$ be a prime ideal in Proj$(S)$ corresponding to a closed point distinct from $x$. Then $L \ne Q$, and to finish the proof, we need only show that $f \not \in L$.    
  As noted above, ${\ff m} \subseteq L$, so if $f \in L$,  then since $h \in {\ff m}S$, we have $g^e \in L$. But then $Q = ({\ff m},g)S \subseteq L$, forcing $Q = L$ since $Q$ is maximal in Proj$(S)$. This contradiction implies that $f \not \in L$, and hence every closed point of $\PPD$ distinct from $x$ is in $(\PPD)_f$, which completes the proof.   
\end{proof}

\begin{remark}  {\em 
 When the valuation rings in $Z$ do not dominate $D$,
the theorem  can still  be applied  if there exists $Y \subseteq \X$ containing $F$ such that (a) each valuation ring in $Y$ other than $F$  dominates $D$, (b) each valuation ring in $Z$ specializes to a valuation ring in $Y$,  and (c) no $D$-morphism $\phi:Y \rightarrow \PPD$ has every closed point in its image. For by the theorem  the holomorphy ring of $Y$ is a  Pr\"ufer domain with torsion Picard group and quotient field $F$.  As an overring of the holomorphy ring of $Y$, the holomorphy ring  of $Z$ has these same properties also.  }
\end{remark}

The following corollary shows how the theorem can be used to prove that real holomorphy rings can be intersected with finitely many non-dominating valuation rings and the result remains a Pr\"ufer domain with quotient field $F$.  In general an intersection of a Pr\"ufer domain and a valuation domain need not be a Pr\"ufer domain. For example, when $D$ is a two-dimensional local Noetherian UFD with quotient field $F$ and $f$ is an irreducible element of $D$, then $D_f$ is a PID and $D_{(f)}$ is a valuation ring, but $D= D_f \cap D_{(f)}$, so that the intersection is not Pr\"ufer. This example can be modified to show more generally that for this choice of $D$, there exist quasicompact schemes in $\X$ that are not affine. 
%

\begin{corollary} Suppose $D$ is essentially of finite type over a real-closed field and that $F$ and  the residue field of $D$  are formally real. Let $H$ be the real holomorphy ring of $F/D$. Then 
for any 
 valuation rings $V_1,\ldots,V_n \in \X$ not dominating $D$,  the ring  $H \cap V_1 \cap \cdots \cap V_n $ is a Pr\"ufer domain with torsion Picard group and quotient field $F$.
\end{corollary}

\begin{proof} Each formally real valuation ring in $\X$ specializes to a formally real valuation ring dominating $D$ (this can be deduced, for example, from \cite[Theorem 23]{Kuh}). Let $Y$ be the set of all the formally real valuation rings dominating $D$,  let $Z = Y \cup \{F,V_1,\ldots,V_n\}$, and let  $\phi:Z \rightarrow \PPD$ be a $D$-morphism.  
Then the image of $Y$ under $\phi$ is contained in $(\PPD)_f$, 
where  $f(T_0,T_1) = T_0^2 + T_1^2$. 
Because $V_1,\ldots,V_n$ do not dominate $D$, they are not mapped by $\phi$ to closed points of $\PPD$.  
Thus the corollary follows from 
Theorem~\ref{P1 theorem}.
 \end{proof}

We include the last corollary as more of a curiosity than an application.  Suppose that $D$ has quotient field $F$. A valuation ring $V$ in $\X$ {\it admits local uniformization} if there exists a projective model $X$ of $F/D$ such that $V$ dominates a regular local ring in $X$. Thus if $\Spec(D)$ has a resolution of singularities, then every valuation ring in $\X$ admits local uniformization. When $D$ is essentially of finite type over a field $k$ of characteristic $0$, then $D$ has a resolution of singularities by the theorem of Hironaka, but when $k$ has positive characteristic, it is not known in general whether local uniformization holds in dimension greater than $3$; see for example \cite{Cut} and \cite{Tem2}. 

\begin{corollary} Suppose that $D$ is a quasi-excellent  integrally closed  local Noetherian domain with quotient field $F$. If there exists a valuation ring in $\X$ that dominates $D$ but does not admit local uniformization, and $Y$ consists of all such valuation rings, 
 then the holomorphy ring  of $Y$ is  a Pr\"ufer domain with torsion Picard group. 
\end{corollary} 

\begin{proof} Let $Z = Y \cup \{F\}$, and let $\phi:Z \rightarrow \PPD$ be a $D$-morphism. Then by Proposition~\ref{factors}, $\phi$ factors through   a  projective model $X$ of $F/D$. Since $Y$ is nonempty, the projective model $X$ has a singularity, and thus since $D$ is quasi-excellent, the singular points of $X$ are contained in a proper nonempty closed subset of $X$. In particular, there are closed points of $X$ that are not in the image of the  domination map $Z \rightarrow X$, and hence there are closed points of $\PPD$ that are not in the image of $\phi$. Therefore, by Theorem~\ref{P1 theorem}, $A$ is a Pr\"ufer domain with torsion Picard group and quotient field $F$. 
\end{proof}

In particular, all the valuation rings that dominate $D$ and do not admit local uniformization lie in an affine scheme in $\X$.

\end{document}